\documentclass{amsart}
\usepackage{amsmath,amssymb,amsthm,amscd,xypic}
\usepackage{color}
\usepackage{enumerate}
\usepackage[T1]{fontenc} 
\usepackage{hyperref}

\title{On a characterization of (co)silting objects}

\author{Simion Breaz}
\address{Babe\c s--Bolyai University, Faculty of Mathematics and Computer Science, Department of Mathematics,
1 Mihail Kog\u alniceanu, 400084 Cluj--Napoca, Romania}
\email{bodo@math.ubbcluj.ro, simion.breaz@ubbcluj.ro}

\date{\today}

\newcommand{\Z}{\mathbb{Z}}

\DeclareMathOperator{\Hom}{Hom}

\newcommand{\A}{\mathcal{A}}
\newcommand{\B}{\mathcal{B}}

\newcommand{\Y}{\mathcal{Y}}
\newcommand{\C}{\mathcal{C}}
\newcommand{\D}{\mathcal{D}}

\newcommand{\U}{\mathcal{U}}
\newcommand{\X}{\mathcal{X}}
\newcommand{\V}{\mathcal{V}}

\newcommand{\add}{\mathrm{add}}
\newcommand{\Add}{\mathrm{Add}}

\newcommand{\Prod}{\mathrm{Prod}}

\newcommand{\thick}{\mathrm{thick}\,}

\newcommand{\codim}{\mathrm{codim}}

\theoremstyle{plain}
\newtheorem{thm}{Theorem}[section]
\newtheorem{lemma}[thm]{Lemma}
\newtheorem{prop}[thm]{Proposition}
\newtheorem{cor}[thm]{Corollary}

\newtheorem{constr}[thm]{Construction}
\newtheorem{thm-intro}{Theorem}

\theoremstyle{definition}

\theoremstyle{remark}
\newtheorem{rem}[thm]{Remark}

\newtheorem{claim}{Claim}

\begin{document}


\begin{abstract} We prove that an object $U$ in a triangulated category with coproducts is silting if and only if it is a (weak) generator of the category, the orthogonal class $U^{\perp_{>0}}$ contains $U$, and $U^{\perp_{>0}}$ is closed under direct sums. The proof can be dualized to provide a characterization for cosilting objects in triangulated categories with products.
\end{abstract}

\subjclass[2010]{18G80, 18E40}

\keywords{triangulated category, t-structure, silting object, cosilting object}

\maketitle


\section{Introduction}

Let $\D$ be a triangulated category with direct sums. If $U\in \D$ and $n$ is an integer, we denote $U^{\perp_{>n}}=\{X\in \D\mid \Hom_\D(U, X[p])=0 \textrm{ for all }p>n\}$. The classes $U^{\perp_{\leq n}}$, $^{\perp_{>n}}U$, $^{\perp_{\leq n}}U$ are defined in the same way. A t-structure in $\D$ is a pair of subcategories $(\U,\V)$ such that $\Hom_\D(\U,\V)=0$ (i.e., $\Hom_\D(U,V)=0$ for all $U\in \U$ and $V\in \V$), $\U[1]\subseteq \U$, $\V[-1]\subseteq \V$, and for every object $X\in \D$ there exists a triangle $U\to X\to V\to U[1]$ with $U\in \U$ and $V\in \V$. In these conditions we have
$\V=\U^{\perp_0}$ and $\U={^{\perp_{0}}\V}$.   

An object $U\in \D$ is called \textit{silting} if $(U^{\perp_{>0}}, U^{\perp_{\leq 0}})$ is a t-structure. Dually, $U$ is called \textit{cosilting} if 
$({^{\perp_{\leq 0}}U},{^{\perp_{>0}} U})$ is a t-structure. We refer to \cite{An-19} for a survey about these objects.

From the above definition, it is easy to see that if $U$ is silting then it has the following properties:
\begin{itemize}
	\item[(S1)] $U\in U^{\perp_{>0}}$,
	\item[(S2)] $U^{\perp_{>0}}$ is closed under direct sums,
	\item[(S3)] $U^{\perp_\Z}=0$.
\end{itemize} 
For many reasonable categories, an object $U$ is silting if and only if it satisfies (S1), (S2), and (S3).
The proof for this characterization is often based on the fact that in these categories the class $ U^{\perp_{\leq 0}}$ is a coaisle of a t-structure e.g., in \cite[Proposition 4.13]{PV}. In fact, this proof can be extended to well-generated triangulated categories, since we know from \cite{Nee18} that in well-generated triangulated categories we always have t-structures of the form $(\overline{\langle U\rangle}^{[-\infty,0]}, U^{\perp_{\leq 0}})$, where $\overline{\langle U\rangle}^{[-\infty,0]}$ is the smallest subcategory that contains $U$, is closed under positive shifts, coproducts and extensions (hence, under the hypotheses (S1) and (S2) we have $\overline{\langle S\rangle}^{[-\infty,0]}\subseteq U^{\perp_{>0}}$).

Neeman's proof presented in \cite{Nee18} is based on the fact that well-generated categories satisfy Brown Representability Theorem. Since it is not known if these categories also satisfy the Brown Representability Theorem for the dual, the above proof cannot be dualized to obtain a similar result for cosilting objects. However, such a characterization is known when ${ ^{\perp_{>0}}U}$ is already a coaisle of a t-structure, e.g. for pure-injective objects in compactly generated categories (see \cite{AMV17} and \cite{HHZ21}). 

The main aim of the present paper is to show, using a proof that can be dualized, that in all triangulated categories with coproducts an object $U$ is silting objects if and only if it satisfies the conditions (S1), (S2), and (S3). Therefore, we also conclude that an object $U$ in a triangulated category with products is cosilting if and only if  
\begin{itemize}
	\item[(C1)] $U\in { ^{\perp_{>0}}U}$,
	\item[(C2)] $^{\perp_{>0}}U$ is closed under direct products,
	\item[(C3)] $^{\perp_\Z}U=0$.
\end{itemize}

All these are consequences of Theorem \ref{thm:main-silting}, where we assume that there is a \textit{cocomplete pre-aisle} subcategory (i.e. it is closed under direct sums, extensions, and positive shifts, cf. \cite[Definition 1.1]{Nee18}) $\A$ such that $U\in \A$ and $\Hom_\D(U,\A[n])=0$ for some positive integer $n$. Moreover, in this case, we can replace the conditions (S1) and (S2) with the hypothesis  
\begin{itemize}
	\item[(S1.5)] $\Add (U)\subseteq U^{\perp_{>0}}$, 
\end{itemize}
where $\Add(U)$ represents the class of all direct summands in direct sums of copies of $U$.

A version of Theorem \ref{thm:main-silting}, when $\A$ is already an aisle of a t-structure, is proved in \cite[Theorem 2]{NSZ}. In the case when $\A$ is the aisle induced by a silting object $V$, we will say that $U$ is \textit{ $n$-$V$-intermediate}. These objects are studied in literature especially when $\D$ is the derived category of a ring and $\A$ is the aisle of the standard t-structure (see \cite{AMV:2015} , \cite{Wei-isr}).   
It is often used (e.g., in \cite[Lemma 5.4]{An-19}) the fact that in the presence of the Brown Representability Theorem, there exists a co-t-structure $(\X, V^{\perp_{>0}})$. For other examples where this co-t-structure is used, we refer to \cite{IJY}. 
We will present in Section \ref{intermediate} some characterizations for the $n$-$V$-intermediate silting objects in triangulated categories with coproducts that can be obtained without using the existence of co-t-structures. The proofs can also be dualized to obtain similar statements for cosilting objects.

In the following, $\D$ will be a triangulated category. All triangles used in this paper will be distinguished triangles. If $A\overset{\alpha}\to B\to X\to A[1]$ is a triangle in $\D$ then the morphism $B\to X$ will be called \textit{a cone} of $\alpha$. We will also use the term \textit{cone} for the object $X$  (that is unique up to isomorphism). 
For other properties valid in triangulated categories, we refer to \cite{Neeman}.

\section{A characterization for (co) silting in categories with direct sums (products)}\label{sect2}

If $\X$ and $\Y$ are classes from $\D$ then 
\begin{align*}\X\ast \Y=\{Z\in \D\mid & \textrm{ there exists a triangle } \\ & X\to Z\to Y\to X[1] \textrm{ with } X\in\X \textrm{ and }Y\in\Y\}.\end{align*}
We recall some basic properties for $\ast$. 

\begin{lemma}\label{lem:compuneri} The following statements are true: 
	\begin{enumerate}[{\rm 1)}] \item The operation $\ast$ is associative.
		
\item If $A\overset{\alpha}\to B\to X\to A[1]$, $B\overset{\beta}\to C\to Y\to B[1]$, and $A\overset{\beta\alpha}\to C\to Z\to A[1]$ are triangles then $Z\in X\ast Y$.

\item If $\X$ and $\Y$ are closed with respect to direct sums then $\X\ast \Y$ is closed under direct sums. 

\item  \cite[Proposition 2.7]{IyY} If $\X$ is closed under finite direct sums and under direct summands, and $\Hom_\D(\X,\X[i])=0$ for all $i=\overline{1,n}$ then $\X\ast\dots\ast \X[n]$ is closed under direct summands.
\end{enumerate} 
\end{lemma}


If $\D$ is a triangulated category with  coproducts, a full subcategory $\A$ of $\D$ is a \textit{cocomplete pre-aisle} if $\A$ is closed under extensions,  direct sums, direct summands, and positive shifts ($\A[1]\subseteq \A$). Consequently, if $A,B\in\A$ and $A\to B\to C\to A[1]$ is a triangle, then $C\in \A$.  

The main aim of this section is to prove the following 

\begin{thm}\label{thm:main-silting}
Let $\D$ be a triangulated category with coproducts. Suppose that $U\in \D$ is an object such that: 
\begin{itemize}
	\item[(S1.5)] $\Add (U)\subseteq U^{\perp_{>0}}$, and
	\item[(S3)] $U^{\perp_{\Z}}=0$.
\end{itemize}

Assume that there exists a cocomplete pre-aisle $\A$ in $\D$ with the following properties: 
\begin{enumerate}[{\rm (1)}]
	\item $U\in \A$, and
	\item there exists a positive integer $n$ such that $\Hom_\D(U, \A[n])=0$.
\end{enumerate}

Then $U$ is a silting object.
\end{thm} 

The proof follows the main ideas from the proofs of \cite[Theorem 3.11]{BHM} and \cite[Theorem 2.3]{PoS}. Recall that, since $\D$ is an additive category with direct sums, for every pair of objects $K,X\in \D$ the canonical morphsim $f:K^{(I)}\to X$, where $I=\Hom_{\D}(K,X)$, is an $\Add(K)$-precover. This means that for every $M\in \Add(K)$ the morphism $\Hom_{\D}(M,f)$ is surjective. We need the following

\begin{constr}\label{constr-sir}
Let $U$ be an object from $\D$. If $X\in \D$, we construct inductively a sequence of morphisms $f_k:X_k\to X_{k+1}$ in the following way:
\begin{enumerate}[{ (i)}]
	\item $X_0=X$;
	\item If $X_k$ is constructed then $f_k:X_k\to X_{k+1}$ will be a cone of an $\Add(U[k])$-preecover $U[k]^{(I_k)}\to X_k$ (it completes the precover to a triangle).  
\end{enumerate} 

For every $i> k$, we consider the morphism $f_{ki}:X_k\to X_{i}$ that are obtained as the composition of the morphisms $f_k,\dots,f_{i-1}$. Moreover, $f_{ii}:X_i\to X_i$, $i\geq 0$, will be the indentity maps. We denote by $S_{ki}$ the cone of $f_{ki}$.
\end{constr}

We have the following properties:

\begin{lemma}\label{lem:basic-inductive}
Let $U$ be an object from $\D$ such that $\Add(U)\subseteq U^{\perp_{>0}}$. For the objects constructed in Construction \ref{constr-sir}, we have the following properties:
\begin{enumerate}[{\rm (a)}]
	\item if $k\geq 0$ then $\Hom_\D(U[j],X_{k+1})=0$ for all $j=\overline{0,k}$;
	\item if $k\leq i$ then $S_{ki}\in  \Add(U[k+1])\ast \dots \ast \Add(U[i])$.
\end{enumerate}	
\end{lemma}

\begin{proof}
(a) For all $k\geq 0$, in the exact sequence of abelian groups \begin{align*}\Hom_\D(U[k],U[k]^{(I_k)})\to \Hom_\D(U[k], X_k) & \to \Hom_\D(U[k], X_{k+1})\\ & \to \Hom_\D(U[k],U[k+1]^{(I_k)}),\end{align*}
the first morphism is surjective, hence the third morphism is zero. But the last group is also equal to zero. It follows that  $\Hom_\D(U[k], X_{k+1})=0$. 

Therefore, the property (a) is true for $k=0$. We will proceed by induction on $k$. Assuming that (a) is valid for $k-1$, that is  $\Hom_\D(U[j],X_{k})=0$ for all $j=\overline{0,k-1}$, we can use the exact sequences \[\Hom_\D(U[j], X_k)  \to \Hom_\D(U[j], X_{k+1}) \to \Hom_\D(U[j],U[k+1]^{(I_k)})\] 
to conclude that $\Hom_\D(U[j], X_{k+1})=0$ for all $j=\overline{0,k-1}$. Using what we already proved in the first part of the proof, we conclude that (a) is valid for $k$. 

(b) follows from Lemma \ref{lem:compuneri}.
\end{proof}

\begin{prop}
Let $U$ be an object from $\D$. Assume that there exists a full subcategory $\A$  cocomplete pre-aisle $\A$ in $\D$ such that 
\begin{itemize}
	\item[{\rm 1)}] $U\in \A$,
	\item[{\rm 2)}] $\Hom_\D(U, \A[n])=0$ for some positive integer $n$,
	\end{itemize}
and that $U$ satisfies the condition
\begin{itemize}
	\item[{\rm (S1.5)}] $\Add (U)\subseteq U^{\perp_{>0}}$.
\end{itemize}

Then for every $X\in \D$ there exists a triangle $Y\to X\to Z\to Y[1]$ such that 
\begin{enumerate}[{\rm (a)}]
	\item $Z\in U^{\perp_{\leq 0}}$,
	\item $Y\in  U^{\perp_{>0}}$,
	\item $\Hom_\D(Y, U^{\perp_{\leq 0}})=0$.
\end{enumerate}
\end{prop}

\begin{proof}
Let $X\in\D$. We consider the system of morphisms from Construction \ref{constr-sir}. For $k,j\geq 0$, we denote by $X^{(j)}_k$ a copy of $X_k$, by $u_k^{(j)}:X_k^{(j)}\to X_k^{(i+1)}$ the identity morphism, and by $f_{ki}^{(j)}:X_k^{(j)}\to X_i$ the morphism $f_{ki}$. We have the equality $f_if_{ki}^{(i)}=f_{k(i+1)}^{(i+1)}u_k^{(i)}$. 

Let $k\geq 0$. Using \cite[Proposition 4.23]{stacks} we observe that there exists a commutative diagram 
\[\xymatrix{ \oplus_{i\geq k}X_k^{(i)}  \ar[rr]^{1-\oplus_{i\geq k}u_k^{(i)}} \ar[d]^{\oplus_{i\geq k}f^{(i)}_{ki}} & &  \oplus_{i\geq k}X_k^{(i)} \ar[d]^{\oplus_{i\geq k}f^{(i)}_{ki}} \ar[r] & X_k\ar[d] \\
\oplus_{i\geq k}X_i  \ar[rr]^{1-\oplus_{i\geq k}f_i}\ar[d] & &  \oplus_{i\geq k}X_i \ar[r]\ar[d] & Z\ar[d] \\
\oplus_{i\geq k}S_{ki}\ar[rr]  & & \oplus_{i\geq k}S_{ki} \ar[r] & C_k 
}\]
such that all lines and columns are triangles. Note that, up to isomorphism, $Z$ does not depends on $k$ since it is the homotopy colimit of the sequence $(f_i)_{i\geq 0}$. 

\begin{claim}\label{claim1}
Suppose that $k>n$ is an integer, and $0\leq \ell\leq k-n$. Then $$\Hom_\D(U,C_k[-\ell])=0.$$
\end{claim}

For every $i\geq k$ we have $S_{ki}\in  \Add(U[k+1])\ast \dots \ast \Add(U[i])\subseteq \A[k+1]$ (recall that $S_{kk}=0$). It follows that for every $\ell$ such that  $0\leq \ell\leq k-n$, we have $\oplus_{i\geq k}S_{ki}[-\ell]\in \A[n]$, hence   $\Hom_\D(U,\oplus_{i\geq k}S_{ki}[-\ell])=0$. Moreover, $\oplus_{i\geq k}S_{ki}[-\ell+1]\in \A[n]$ since $\A$ is closed under positive shifts. 
Therefore, $\Hom_\D(U,\oplus_{i\geq k}S_{ki}[-\ell+1])=0$. It follows that the claim is true.


\begin{claim}\label{claim2}
$\oplus_{i\geq 0}S_{0i}[-1]\in U^{\perp_{> 0}}$.	
\end{claim}

As before, for every $i\geq n$ we have $S_{ni}\in\A[n+1]$, hence $\oplus_{i\geq n}S_{ni}\in\A[n+1]$. Since $\A[n]$ is closed with respect to positive shifts, it follows that 
$$\Hom_\D(U, \oplus_{i\geq n}S_{ni}[-1][p+1])=\Hom_\D(U, \oplus_{i\geq n}S_{ni}[p])=0 \textrm{ for all }p\geq 0,$$
hence $\oplus_{i\geq n}S_{ni}[-1]\in U^{\perp_{> 0}}$.

Moreover, for every $i\geq n$ we have a triangle $$S_{0n}\to S_{0i}\to S_{ni}\to S_{0n}[1].$$ The direct sum of these triangles induces a triangle  
$$ \oplus_{i\geq n} S_{0n}\to \oplus_{i\geq n} S_{0i}\to \oplus_{i\geq n} S_{ni}\to \oplus_{i\geq 0}S_{0n}[1].$$
Observe that for every $i>0$ we have $$S_{0i}[-1]\in \Add(U)\ast \dots \ast \Add(U[i-1]),$$ and this class is closed under direct sums. Moreover, since all the classes $\Add(U)$, $\dots$, $\Add(U[i-1])$ are contained in $U^{\perp_{>0}}$, we obtain $$\Add(U)\ast \dots \ast \Add(U[i-1])\subseteq U^{\perp_{>0}}.$$ 

In particular, for $i=n$ it follows that $$\oplus_{i\geq n} S_{0n}[-1]\in \Add(U)\ast \dots \ast \Add(U[n-1])\subseteq U^{\perp_{>0}},$$ and using the above triangle, we obtain $ \oplus_{i\geq n}S_{0i}[-1]\in U^{\perp_{>0}}$. 

Since $S_{00}=0$, we also have  $\oplus_{0\leq i< n}S_{0i}[-1]\in  U^{\perp_{> 0}}$, so
 the proof of the claim is complete.


\medskip

We will prove that $Z$ verify the condition (a) and that  $Y=C_0[-1]$ verifies (b) and (c).


If we consider an integer
$\ell\geq 0$, and we take $k=n+1+\ell$, we apply Lemma \ref{lem:basic-inductive} to observe that $\Hom_\D(U,X_k[-\ell])=0$. 
From the triangle $$X_k\to Z\to C_k\to X_k[1]$$ and Claim \ref{claim1} we conclude that $\Hom_\D(U,Z[-\ell])=0$. Since $\ell$ was chosen arbitrarily, it follows that $Z\in U^{\perp_{\leq 0}}$.

From Claim \ref{claim2} it follows that $Y=C_0[-1]$ verifies (b). For the property (c), we first observe that $U, U[1], \dots, U[i]\in {^{\perp_0}(U^{\perp_{\leq 0}})}$. The class ${^{\perp_0}(U^{\perp_{\leq 0}})}$ is closed under direct sums, direct summands, and extensions. Since $$S_{0i}\in  \Add(U[1])\ast \dots \ast \Add(U[i]),$$ we obtain the equalities  $$\Hom_\D(\oplus_{i\geq 0}S_{0i},U^{\perp_{\leq 0}})=0$$ and $$\Hom_\D(\oplus_{i\geq 0}S_{0i}[-1],U^{\perp_{\leq 0}})=0.$$ Using the bottom triangle from the commutative diagram presented in the beginning of the proof, it follows that $\Hom_\D(Y, U^{\perp_{\leq 0}})=0$.
%
\end{proof}

In order to complete the proof of Theorem \ref{thm:main-silting}, we only need to observe, using (S3), that $U^{\perp_{>0}}\cap U^{\perp_{\leq 0}}=0$, and to apply the following

\begin{lemma}\label{lem:t-str}
Let $\D$ be a triangulated category. If $\U$ and $\V$ are subcategories in $\D$ such that $\U\cap \V=0$, $\U[1]\subseteq \U$, $\V[-1]\subseteq \V$, and for every $X\in \D$ there exists a triangle $U\to X\to V\to U[1]$ such that $U\in \U$, $V\in \V$, and $\Hom(U,\V)=0$ then $(U,V)$ is a t-structure.	
\end{lemma}

\begin{proof}
Let $X\in \U$, and consider the triangle $U\to X\to V\to U[1]$ as before. Then $V\in \V\cap \U$, hence $V=0$. Hence $U\cong X$, so $\Hom(X,\V)=0$. 
\end{proof}	

Now, we can prove that silting objects in triangulated categories with coproducts are characterized by the conditions (S1), (S2) and (S3). 

\begin{cor}\label{cor:car-silt}
Assume the $\D$ is a triangulated category with coproducts. An object $U\in\D$ is silting if and only if the following conditions are true:
\begin{itemize}
	\item[(S1)] $U\in U^{\perp_{>0}}$;
	\item[(S2)] $U^{\perp_{>0}}$ is closed under direct sums;
	\item[(S3)] $U^{\perp_\Z}=0$.
\end{itemize}
\end{cor}

\begin{proof}
In Theorem \ref{thm:main-silting}, we take $\A=U^{\perp_{>0}}$. Since $\Hom_{\D}(U,U^{\perp_{>0}}[1])=0$, we have the conclusion.
\end{proof}

As we already observed, all the above proofs can be dualized to obtain similar results for cosilting objects in triangulated categories with products. For the reader's convenience, we state the main dual results. Here $\Prod(U)$ denoted the class of all direct summand in direct products of copies of $U$.

\begin{thm}\label{thm:main-cosilting}
	Let $\D$ be a triangulated category with products. Suppose that $U\in \D$ is an object such that: 
	\begin{itemize}
		\item[(C1.5)] $\Prod (U)\subseteq {^{\perp_{>0}}U}$, and
		\item[(C3)] ${^{\perp_{\Z}}U}=0$.
	\end{itemize}
	Assume that there exists a complete pre-coaisle $\B$ {\rm (i.e., $\B$ is closed under extensions, direct products, direct summands and negative shifts)} in $\D$ with the following properties: 
	\begin{enumerate}[{\rm (1)}]
		\item $U\in \B$, and
		\item there exists a positive integer $n$ such that $\Hom_\D(\B[-n],U)=0$.
	\end{enumerate}
	
	Then $U$ is a cosilting object.
\end{thm}

\begin{cor}\label{cor:car-cosilt}
	Assume the $\D$ is a triangulated category with coproducts. An object $U\in\D$ is cosilting if and only if the following conditions are true:
	\begin{itemize}
		\item[(C1)] $U\in { ^{\perp_{>0}}U}$,
		\item[(C2)] $^{\perp_{>0}}U$ is closed under direct products,
		\item[(C3)] $^{\perp_\Z}U=0$.
	\end{itemize}
\end{cor}

\section{Intermediate silting objects}\label{intermediate}

In this section we will assume that $\D$ has coproducts and that $V$ is a silting object. Let $n$ be a positive integer, we will say that a silting object $U\in \D$ is \textit{$n$-$V$-intermediate} if $$V[n]\in U^{\perp_{>0}} \textrm{ and }U\in V^{\perp_{>0}}.$$
For instance, the bounded silting complexes of $R$-modules that are studied in \cite{AMV:2015} and \cite{Wei-isr} coincide with the $n$-$R$-intermediate silting objects from the (unbounded) derived category $\mathbf{D}(R)$. They also appear in \cite[Theorem 2]{NSZ}. In Theorem \ref{thm:inter-n-silt} we will provide a characterization that extends the characterizations of tilting modules and bounded silting complexes presented in \cite{Wei-isr} and \cite{Ba}. The proofs can be dualized to extend the similar results proved for cotilting modules and cosilting complexes in \cite{Ba} and \cite{ZW2}.

If $\C$ is a class of objects in $\D$,  we say that  $X\in \D$ has the {\em $\C$-dimension } (respectively {\em $\C$-codimension }) {\textit at most} $n$, and we write
$\dim_\C X\leq n$, ($\codim_\C X\leq n$)  provided that there  is a sequence of triangles
\[ X_{i+1}\to C_i\to X_i\to X_{i+1}[1]\hbox{ with }0\leq i< n \]
\[\hbox{(respectively } X_i\to C_i\to X_{i+1}\to X_{i}[1]\hbox{ with }0\leq i< n\hbox{)}\] in $\D$, such that
$C_i\in\C$, $X_0=X$ and $X_{n+1}\in \C$. We will write $\dim_\C X< \infty$ ($\codim_\C X<\infty$) if we can find a positive integer $n$ such that $\dim_\C X\leq n$ (respectively, $\codim_\C X\leq n$).

\begin{lemma} \cite[Lemma 3.8]{An-19} \label{lem:dim-codim}
	Suppose that $\C$ is a class in a triangulated category $\D$, and $X$ is an object from $\D$. Then
	\begin{enumerate}[{\rm i)}]
		\item $\dim_\C X\leq n$ if and only if $X\in \C\ast \C[1]\ast \dots \ast \C[n]$;
		\item $\codim_\C X\leq n$ if and only if $X\in \C[-n]\ast \C[-n+1]\ast \dots \ast \C$, if and only if $\dim_\C X[n]\leq n$;
		\item $\dim_\C X\leq n$ if and only if $\codim_{\C[n]}X\leq n$.
	\end{enumerate}
\end{lemma}

\begin{rem}
	Since the direct sums preserve the exactness of the triangles in $\D$, if we assume that $\C$ is closed under (finite) direct sums the the class of objects of $\C$-(co)dimension at most $n$ is closed under  (finite) direct sums. Moreover, if $\C=\add \C$ and $\Hom_{\D}(\C,\C[1])=0$ we can use Lemma \ref{lem:compuneri} to observe that the class of objects of $\C$-(co)dimension at most $n$ is closed under direct summands.  
\end{rem}

\begin{lemma}\label{lem:dim-codim-perp}
	Let $U$ and $V$ be objects from $\D$. 
	\begin{enumerate}[{\rm (i)}]
		\item If $\codim_{\Add(U)}V\leq n$ then $U^{\perp_{>0}}\subseteq V^{\perp_{>0}}$.
		\item If $\dim_{\Add V}U\leq n$ then $V^{\perp_{>0}}[n]\subseteq U^{\perp_{>0}}$ (or, equivalently,  $V^{\perp_{>0}}\subseteq U^{\perp_{>n}}$). 
	\end{enumerate}
\end{lemma}

\begin{proof}
	(i) We have the triangles: 
	\begin{align*}
		& V\to U_0\to K_1\to V[1], \\
		& K_1\to U_1\to K_2\to K_1[1], \\ 
		& \dots \\
		& K_{n-1}\to U_{n-1}\to K_n\to K_{n-1}
	\end{align*}
	such that $U_0,\dots,U_{n-1},K_n\in \Add(U)$. 
	
	If $X\in U^{\perp_{>0}}$, we look at the long exact sequences 
	\begin{align*}
		 \Hom_{\D}(K_n,X[1])\to &\Hom_{\D}(U_{n-1},X[1])\to  \Hom_{\D}(K_{n-1},X[1])\to \\  &\Hom_{\D}(K_n,X[2])\to \Hom_{\D}(U_{n-1},X[2])\to \dots , \\
		\dots & \\
		\Hom_{\D}(K_2,X[1])\to & \Hom_{\D}(U_{1},X[1])\to \Hom_{\D}(K_{1},X[1])\to  \\ & \Hom_{\D}(K_2,X[2])\to \Hom_{\D}(U_{1},X[2])\to \dots , \\
		 \Hom_{\D}(K_1,X[1])\to & \Hom_{\D}(U_{0},X[1])\to \Hom_{\D}(V,X[1])\to  \\  & \Hom_{\D}(K_1,X[2])\to  \Hom_{\D}(U_{0},X[2])\to \dots , \\
	\end{align*}
	to conclude inductively that $X\in K_{i}^{\perp_>0}$ for all $0\leq i<n$, where $K_0=V$. 
	
	(ii) This can be proved in the same way, or we can use (i) together with Lemma \ref{lem:dim-codim}. 
\end{proof}

We say that an object $X$ is \textit{$n$-$V$-presented by $U$} if for every $0\leq i< n$ there exists a triangle 
$N_{i+1}\to U_i\to N_i\to N_{i+1}[1]$ such that $N_0=X$, for all $i>0$ we have $N_i\in V^{\perp_{>0}}$, and all $U_i$ are from $\Add(U)$. We denote by $\mathrm{Pres}_{V}^n(U)$ the class of all objects from $\D$ that are $n$-$V$-presented by $U$. If $U\in V^{\perp_{>0}}$,  it is enough to verify that $N_{n}\in V^{\perp_{>0}}$ to conclude that $N_i\in V^{\perp_{>0}}$ for all $i=\overline{1,n}$.

\begin{thm}\label{thm:inter-n-silt}
	Suppose that $V$ is silting. The following are equivalent for an object $U\in \D$ and a positive integer $n$:
	\begin{enumerate}[{\rm a)}]
		\item $U$ is an $n$-$V$-intermediate silting object;\\
		\item \begin{itemize}
			\item[(I1)] $U^{\perp_{>0}}\subseteq V^{\perp_{>0}}$,
			\item[(I2)] $V^{\perp_{>0}}[n]\subseteq U^{\perp_{>0}}$,
			\item[(S1.5)] $\Add(U)\subseteq U^{\perp_{>0}}$;\\
		\end{itemize} 
	\item \begin{itemize}
			\item[(I2)] $V^{\perp_{>0}}[n]\subseteq U^{\perp_{>0}}$,
			\item[(I3)] $U\in V^{\perp_{>0}}$,
		\item[(S1.5)] $\Add(U)\subseteq U^{\perp_{>0}}$,
			\item[(S3)] $U^{\perp_\Z}=0$;\\
	\end{itemize} 
		\item \begin{itemize}
			\item[(I3)] $U\in V^{\perp_{>0}}$, 
			\item[(I4)] $\mathrm{Pres}_{V}^n(U)=U^{\perp_{>0}}$;\\	
		\end{itemize}
	\item  \begin{itemize}
		\item[(I5)] $\dim_{\Add V}U\leq n$,
		\item[(I6)] $\codim_{\Add(U)}V\leq n$,
		\item[(S1.5)] $\Add(U)\subseteq U^{\perp_{>0}}$.\\
			\end{itemize}
	\end{enumerate}	
Under these conditions, $U$ is a silting object.
\end{thm} 

\begin{proof}
a)$\Rightarrow$b) Since $U$ is a silting object, the condition (I3) is automatically satisfied. Moreover, $U^{\perp_{>0}}$ is the smallest cocomplete pre-aisle that contains $U$, hence (I1) is true. In a similar way, we use $V[n]\in U^{\perp_{>0}}$ to obtain (I2). 	

b)$\Rightarrow$c) We only have to prove (S3).
Let $X\in  U^{\perp_{\Z}}$. Then for all integers $i$ we have $X[i]\in  U^{\perp_{>0}}$. From (I1) we obtain that for all $i\in \Z$ we have $X[i]\in  V^{\perp_{>0}}$, hence $X\in  V^{\perp_{\Z}}=0$.  

c)$\Rightarrow$a) It is enough to prove that $U$ is a silting object, so we will apply Theorem \ref{thm:main-silting} for $\A=V^{\perp_{>0}}$.	
	
c)$\Rightarrow$d) 
For (I4), we remind that for every $X\in \D$ there is a triangle $$Y\to U^{(I)}\overset{f}\to X\to Y[1]$$ such that $f$ is an $\Add(U)$-precover for $X$.  
In particular, $\Hom(U,f)$ is surjective. If $X\in U^{\perp_{>0}}$, it follows by (I3) that $Y\in  U^{\perp_{>0}}$. Then $X\in \mathrm{Pres}_{U}^k(U)$ for all $k\geq 0$. 
From (I1), we obtain $\mathrm{Pres}_{U}^k(U)\subseteq \mathrm{Pres}_{V}^k(U)$ for all $k\geq 0$, hence $$\textstyle U^{\perp_{>0}}\subseteq \bigcap_{k>0}\mathrm{Pres}_{V}^k(U)\subseteq \mathrm{Pres}_{V}^n(U).$$ 
	
Conversely, let $X\in \mathrm{Pres}_{V}^n(U)$. We consider a family of triangles $$N_{i+1}\to U_i\to N_i,\ 0\leq i\leq n,$$ such that $N_0=X$, for all $i>0$ we have $N_i\in V^{\perp_{>0}}$, and all $U_i$ are from $\Add(U)$. From the condition (I2), it follows that $N_n\in U^{\perp_{>n}}$. Therefore, for every $k>0$ and every $0\leq i\leq n$ we have $$\Hom_\D(U,X[k])\cong \Hom_\D(U, N_i[i+k])\cong \Hom_\D(U,N_n[n+k])=0,$$ hence $X\in U^{\perp_{>0}}$.
	
d)$\Rightarrow$b) From $\Add(U)\subseteq \mathrm{Pres}_{V}^n(U)$, we obtain the condition (S1.5). Moreover, from (I3) it follows that $\Add(U)\subseteq V^{\perp_{>0}}$.

	
Let $X\in U^{\perp_{>0}}= \mathrm{Pres}_{V}^n(U)$. There exists a triangle $Y\to M\to X\to Y[1]$ such that $M\in \Add(U)\subseteq V^{\perp_{>0}}$ and $Y\in V^{\perp_{>0}}$. Since $V^{\perp_{>0}}$ is closed under positive shifts and extensions, we obtain $X\in V^{\perp_{>0}}$. It follows that $U^{\perp_{>0}}\subseteq V^{\perp_{>0}}$, hence the condition (I1) is satisfied. 
	
In order to prove (I2), we take an object $X\in V^{\perp_{>0}}$, and in the sequence of triangles $$X[k-1]\to 0\to X[k]\overset{=}\to X[k] ,\ 1\leq k< n,$$ we denote $X[k]=N_{n-k}$ and $X[k-1]=N_{n-k+1}$ to conclude that $X[n]\in \mathrm{Pres}_{V}^n(U)$. Then  $ V^{\perp_{>0}}[n]\subseteq U^{\perp_{>0}}$, and the proof is complete.
	%
	
d)$\Rightarrow$e) From the proof of c)$\Rightarrow$d), we know that $U^{\perp_{>0}}\subseteq \mathrm{Pres}_{U}^k(U)$ for all $k\geq 0$. Moreover, we know from (I2) that $V^{\perp_{>0}}[n]\subseteq U^{\perp_{>0}}$. Therefore $V[n]\in  U^{\perp_{>0}}\subseteq \mathrm{Pres}_{U}^n(U)$, and there exists a sequence of triangles $$N_{i+1}\to U_i\to N_i\to N_{i+1}[1],\ 0\leq i<n,$$ such that $N_0=V[n]$, for all $i>0$ we have $N_i\in U^{\perp_{>0}}$, and all $U_i$ are from $\Add(U)$. We will prove that $N_n\in \Add(U)$. 

Let $X\in U^{\perp_{>0}}$. Since $U^{\perp_{>0}}\subseteq V^{\perp_{>0}}$, we obtain $\Hom_\D(V[-1],X)=0$.  For every $p>0$ we have $\Hom_\D(U_i[-p],X)=0$, hence $$\Hom_\D(N_{i+1}[-p],X)\cong \Hom_\D(N_i[-p-1],X)$$ for all $0\leq i<n$. We obtain the sequence of isomorphisms \begin{align*}
\Hom_\D(N_{n}[-1],X)& \cong \Hom_\D(N_{n-1}[-2],X)\cong \dots \cong \Hom_\D(N_{0}[-n-1],X) \\ & =\Hom_\D(V[-1],X)=0.\end{align*}

It follows that $\Hom_\D(N_{n}[-1],X)=0$ for all $X\in U^{\perp_{>0}}$.
We also have $N_n\in U^{\perp_{>0}}$, so we can apply \cite[Proposition 4.8]{An-19} to conclude that $N_n\in \Add(U)$. Then $\dim_{\Add(U)}V[n]\leq n$, hence $\codim_{\Add(U)}V\leq n$. 

To prove (I6), it is enough to observe that $U$ is silting and $V[n]$ is an $n$-$U$-intermediate silting object. Therefore, we can apply what we just proved for the $n$-$V$-intermediate silting object $U$. 

e)$\Rightarrow$b) This can be obtained by applying Lemma \ref{lem:dim-codim-perp}.
\end{proof}

If $\C$ is class of objects in $\D$, we denote by $\thick(\C)$ the smallest subcategory of $\D$ that contains $\C$ and is closed under shifts, extensions and direct summands.

\begin{cor}\label{cor-shifted-silting}
	Suppose the $V$ is silting. The following are equivalent for an object $U$:
	\begin{enumerate}[{\rm (i)}]
		\item there exist integers $\ell$ and $n$, $n\geq 0$, such that $U[\ell]$ is a silting $n$-$V$-intermediate object;
		\item $\Add(U)\subseteq U^{\perp_{>0}}$, and $\thick{\Add(U)}=\thick{\Add (V)}$.
	\end{enumerate}	
\end{cor}

\begin{proof}
	(i)$\Rightarrow$(ii) This follows from Theorem \ref{thm:inter-n-silt} and \cite[Lemma 3.16]{NS}.
	
	(ii)$\Rightarrow$(i) From (ii) it follows that the class $\Add(U)$ is silting in $\Add(V)$. From \cite[Proposition 4.3]{An-19} it follows that there exists $\ell$ such that $V\in \Add(U)[-\ell]\ast \dots\ast \Add(U)[\ell]$, hence $V\in \codim_{\Add(U)[-\ell]}V<\infty$ (or, $V[\ell]\in U^{\perp_{>0}}$). A similar argument assures that we can chose $\ell$ such that $\dim_{\Add V}U[-\ell]<\infty$, and we can apply Theorem \ref{thm:inter-n-silt} to obtain the conclusion.  
\end{proof}

\begin{rem}
All the results from this section can be dualized. More precisely, if $\D$ is a category with direct products, and $V$ is a cosilting object, we will say that $U$ is an \textit{$n$-$V$-intermediate cosilting object} if $U$ satisfies the hypotheses of Theorem \ref{thm:main-cosilting} for $\B={^{\perp_{>0}}V}$. Theorem \ref{thm:inter-n-silt} can be dualized to obtain a characterization for these objects. Also, Corollary \ref{cor-shifted-silting} can be dualized.
\end{rem}


 \subsection*{Acknowledgements}
I would like to thank to Michal Hrbek. The results from Section \ref{sect2} were carried out on the strenght of the conversations we had.   
 
The research of S.\ Breaz is supported by a grant of the Ministry of Research, Innovation and Digitization, CNCS/CCCDI--UEFISCDI, project number PN-III-P4-ID-PCE-2020-0454, within PNCDI III.


\end{document}